\documentclass[12pt]{amsart}
\usepackage[utf8]{inputenc}

\usepackage{amssymb,amsthm,amsfonts,enumerate, epsfig,epstopdf,caption, graphicx, xcolor}
\usepackage{amsmath}
\usepackage{subcaption}
\usepackage{hyperref}
\newcommand{\Q}{{\mathbb{Q}}}
\newcommand{\R}{{\mathbb{R}}}
\newcommand{\Z}{{\mathbb{Z}}}
\newcommand{\N}{{\mathbb{N}}}
\newcommand{\C}{{\mathbb{C}}}
\newcommand{\X}{{\mathcal{X}}}
\newcommand{\sltwo}{\mathrm{SL}_2(\mathbb{C})}
\DeclareMathOperator{\tr}{\mathrm{Tr}}

\newtheorem{lemma}{Lemma}[section]
\newtheorem{proposition}[lemma]{Proposition}
\newtheorem{theorem}[lemma]{Theorem}

\newtheorem{definition}[lemma]{Definition}

\title[Seifert manifolds with non-reduced character scheme]{Small Seifert $3$-manifolds with non-reduced $\mathrm{SL}_2(\C)$-character scheme}
\author{Renaud Detcherry}
\address[]{Université Bourgogne Europe, CNRS, IMB UMR 5584, F-21000 Dijon, France \& Institut Universitaire de France}
\email{renaud.detcherry@u-bourgogne.fr}
\thanks{}

\begin{document}
	\maketitle
	\begin{abstract} We complete the work started in \cite{DKS2}, and give a complete description of the $\mathrm{SL}_2(\C)$-character scheme $\X(M)$ of all small Seifert $3$-manifolds $M$. We find that $\X(M)$ is reduced if and only if $M$ admits no exceptional abelian character, and that exceptional abelian characters have multiplicity $2$ in $\X(M).$ 
	\end{abstract}

	\section{Introduction}
	For $M$ a closed connected compact oriented $3$-manifold, let $\mathcal{R}(M)$ be its $\mathrm{SL}_2(\C)$-representation scheme, and let
	$$\X(M)=\mathcal{R}(M)//\mathrm{SL}_2(\C)$$
	denote the $\mathrm{SL}_2(\C)$-character scheme of $M.$ The underlying algebraic set to $\mathcal{R}(M)$ is the representation variety $R(M)=\mathrm{Hom}(\pi_1(M),\mathrm{SL}_2(\C)),$ and the underlying algebraic set to $\X(M)$ is
	$$X(M)=\mathrm{Hom}(\pi_1(M),\mathrm{SL}_2(\C))//\mathrm{SL}_2(\C),$$  which is called the character variety of $M.$ Points in $X(M)$ are in one to one correspondance with equivalence classes of $\mathrm{SL}_2(\C)$-representations of $\pi_1(M),$ where two representations $\rho,\rho'$ are equivalent if and only if they have same trace. The precise definition of $\mathcal{R}(M)$ and $\X(M)$ are more technical, but can be found for instance in \cite[Section 3]{HP23}.
	
	The present note is concerned with describing for which small Seifert manifolds the scheme $\X(M)$ is not reduced, that is, the coordinate ring $\C[\X(M)]$ has non-zero nilpotent elements.
	
	Let $M=S^2(q_1/p_1,q_2/p_2,q_3/p_3)$ be a Seifert manifold that fibers over $S^2$ with $3$ exceptional fibers of order $p_1,p_2,p_3,$ and with Euler number $$e(M)=\frac{q_1}{p_1}+\frac{q_2}{p_2}+\frac{q_3}{p_3}.$$
	In \cite{DKS2}, the authors found that for a closed Seifert manifold $M,$ the character variety $X(M)$ is $0$-dimensional if and only if $M$ fibers over $S^2$ with at most $3$ exceptional fibers and non-zero Euler number, or if $M=\R P^3 \# \R P^3.$ Moreover, recall \cite[Theorem 6.1]{JN83} that for $M=S^2(q_1/p_1,q_2/p_2,q_3/p_3),$ a presentation of $\pi_1(M)$ is:
	\begin{equation}\label{eq:pres}\pi_1(M)=\langle h,c_1,c_2,c_3 \ | \  [c_i,h]=1, c_i^{p_i}h^{q_i}=1,  c_1c_2c_3=1\rangle.\end{equation}
	For $N$ a compact connected $3$-manifold, a point in $X(N)$ is called an \textit{abelian character} if it is the trace of an abelian $\sltwo$-representation of $\pi_1(M).$ Similarly, it is called an \textit{irreducible character} if it is the trace of an irreducible representation.
	Following \cite{DKS2}, we call an abelian character $[\rho]$ of $M=S^2(q_1/p_1,q_2/p_2,q_3/p_3)$ \textit{exceptional abelian} if one has
	$$\tr \rho(h)=\pm 2, \ \textrm{and} \ \tr \rho(c_i)\neq \pm 2 \ \textrm{for} \ i=1,2,3.$$
	
	We note that, as explained in \cite[Proposition 5.8]{DKS2}, any irreducible character of $M$ also satisfies those conditions; hence the exceptional abelian characters are in some sense the abelian characters that look similar to irreducible characters. The following summarizes the results of \cite[Section 5]{DKS2}:
	
	\begin{theorem}\label{thm:DKS2} \cite{DKS2} Let $M=S^2(q_1/p_1,q_2/p_2,q_3/p_3)$ be a Seifert manifold and assume $e(M)\neq 0.$ Then 
		$$|X(M)|=p_1^+ p_2^+ p_3^-  + p_1^- p_2^- p_3^- + \frac{1}{2}(|H_1(M,\Z)| +|H_1(M,\Z/2\Z)|)-x_M$$
		where $p_i^+=\lceil \frac{p_i}{2} \rceil-1$ and $p_i^-=\lfloor \frac{p_i}{2} \rfloor$ and $x_M$ is the number of exceptional abelian characters. Moreover, a point $[\rho]$ in $X(M)$ is reduced, except possibly if $[\rho]$ is an exceptional abelian character. 
	\end{theorem}
	
	We note that \cite[Section 5]{DKS2} also provides formulas for $x_M$ and $|H_1(M,\Z)|$ and $|H_1(M,\Z/2\Z)|,$ which we will not recall here.
	The goal of this note is to prove the following, which completes the description of the character schemes of small Seifert manifolds:
	
	\begin{theorem}\label{thm:main} Let $M=S^2(q_1/p_1,q_2/p_2,q_3/p_3)$ be a Seifert manifold and assume $e(M)\neq 0,$ and let $[\rho]$ be an exceptional abelian character of $M.$ Then $[\rho]$ has multiplicity $2$ in $\X(M).$ In particular, one has
		$$\dim \C[\X(M)]=p_1^+ p_2^+ p_3^-  + p_1^- p_2^- p_3^- + \frac{1}{2}(|H_1(M,\Z)| +|H_1(M,\Z/2\Z)|),$$
		and furthermore,  $\X(M)$ is reduced if and only if $M$ has no exceptional abelian character.
	\end{theorem}
	
	Theorem \ref{thm:main} in particular gives explicit examples of $3$-manifolds with non-reduced character scheme. While many examples of $3$-manifolds with non-reduced character schemes were constructed by Kapovich and Milson \cite{KM17}, to the author's knowledge, no example with $M$ a rational homology sphere, or such that the non-reducedness occurs at an isolated abelian character can be found in the litterature\footnote{We note that \cite[Example 54]{HP23} provides a component of non-reduced characters that contains abelian characters in the character variety of some graph manifold with toroidal boundary. We thank the referee for pointing out that example to us.}. We note that the latter is significant, because while non-reducedness at irreducible characters can be checked by a computation of twisted homology group $H^1(M,\mathrm{Ad}\rho),$ which coincides with $T_{[\rho]}\X(M)$ when $\rho$ is irreducible (see \cite{LM:book}), there is no simple formula for the tangent space at abelian characters.
	
	Our approach to prove Theorem \ref{thm:main} is to use the presentation of $\pi_1(M)$ given in Equation \ref{eq:pres}, together with a theorem of Heusener and Porti \cite{HP23} that, for a finitely generated group $\Gamma,$ gives a presentation of $\C[\X(\Gamma)]$ in terms of a presentation of the group $\Gamma.$ Differentiating the presentation of $\C[\X(M)]$ at an exceptional abelian character $[\rho]$ will give a presentation of the tangent space $T_{[\rho]}\X(M),$ which will be of dimension $1.$ A further differentiation will show that there is no order $2$ deformation of $[\rho].$
	
	Let us note that for a $3$-manifold $M,$ by a theorem of Przytycki and Sikora \cite{PS00}, the coordinate ring $\C[\X(M)]$ is isomorphic to $S_{-1}(M),$ the Kauffman bracket skein module at $A=-1$ of $M.$ Moreover, under additional hypothesis on $M,$ the dimension of $\C[\X(M)]$ is equal to the dimension over $\Q(A)$ of $S(M,\Q(A)),$ the Kauffman bracket skein module of $M$ with $\Q(A)$ coefficients, by a theorem of the author and Kalfagianni and Sikora \cite{DKS1}, see also \cite{FKBT} for a generalization. 
	
	The main theorem in particular completes the computation of the dimensions over $\Q(A)$ of the Kauffman bracket skein modules of small Seifert manifolds that was the main topic of \cite{DKS2}. Furthermore, together with the work of \cite{FKBT}, it shows hat those dimensions can really see the multiplicity of characters in $\mathcal{X}(M).$ Although we will not use anything about Kauffman bracket skein modules,  this was one of the main motivations for the author to write this note.
	
	\textbf{Acknowledgements:} The author thanks L\'eo B\'enard, Charlie Frohman, Effie Kalfagianni and Adam Sikora for their interest and helpful conversations. The IMB receives support from the EIPHI Graduate School (contract ANR-17-EURE-0002).

	\section{Useful formulas}
	\label{sec:formulas}
	
	\subsection{Chebyshev polynomials}
	\label{sec:Cheby-poly}
	\begin{definition}\label{def:chebyshev-pol} For $n\geq 0,$ define two sequences of polynomials $T_n,S_n \in \Z[X]$ by:
		$$T_0=2, T_1=X, T_{n+1}=XT_n -T_{n-1},$$
		$$S_0=0, S_1=1, S_{n+1}=XS_n -S_{n-1}.$$
		The polynomials $T_n$ (resp $S_n$) are called Chebyshev polynomials of the first kind (resp. of the second kind).
	\end{definition}
	We note that the recurrence relation actually allows one to define $S_n$ and $T_n$ for negative $n$ as well, in which case we get $T_{-n}=T_n$ and  $S_{-n}=-S_n.$
	We will make use of the following formulas:
	\begin{lemma}\label{lemma:Chebyshev}
		Let $n\geq 1$ be an integer and let $\zeta$ be a primitive $n$-th root of $1.$ We have the following identities:
		\begin{equation}
		\label{eq:Cheb1} \forall x \in \C^*, \ T_n(x+x^{-1})=x^n+x^{-n} \ \textrm{and} \ S_n(x+x^{-1})=\frac{x^{n}-x^{-n}}{x-x^{-1}}
		\end{equation}
		\begin{equation}
		\label{eq:Cheb2}  T_n'(X)=nS_n(X)
		\end{equation}
		\begin{equation}
		\label{eq:Cheb3} \forall x \in \C^*, \ S_n'(x+x^{-1})=\frac{n(x^n+x^{-n})}{(x-x^{-1})^2}-\frac{(x^n-x^{-n})(x+x^{-1})}{(x-x^{-1})^3}
		\end{equation}
	\end{lemma}
	\begin{proof}
		The first identity is well-known and easily obtained by induction from the definition of $T_n$ and $S_n.$ The second identity is obtained by differentiating the equation $T_n(x+x^{-1})=x^n+x^{-n}$ with respect to $x,$ which gives $T_n'(x+x^{-1})=n\frac{x^n-x^{-n}}{x-x^{-1}},$ and noting that this equation caracterizes $nS_n.$ Similarly, the third equation is obtained by differentiating the equation $S_n(x+x^{-1})=\frac{x^n-x^{-n}}{x-x^{-1}}.$ 
	\end{proof}
	Finally, we will need to consider special values of Chebyshev polynomials and their derivatives:
	\begin{lemma}
		\label{lemma:Chebyshev2} Let $n\in \N^*,$ $\varepsilon\in \lbrace \pm 1 \rbrace$ and $x=\zeta + \zeta^{-1}$ where $\zeta \neq \pm 1$ and $\zeta^n=\varepsilon.$ Then $S_n(x)=0, S_{n+1}(x)=\varepsilon, S_{n-1}(x)=-\varepsilon$ and $S_n(2\varepsilon)=n\varepsilon^{n-1}.$
		Moreover, $S_n'(x)=\frac{2n\varepsilon}{x^2-4}$ and $S_{n\pm 1}(x)=\frac{nx\varepsilon}{x^2-4}.$
	\end{lemma}
	\begin{proof}
		We have 
		$$S_n(x)=S_n(\zeta+\zeta^{-1})=\frac{\zeta^n -\zeta^{-n}}{\zeta-\zeta^{-1}}=0,$$
		$$S_{n+1}(x)=S_{n+1}(\zeta+\zeta^{-1})=\frac{\zeta^{n+1} -\zeta^{-n-1}}{\zeta-\zeta^{-1}}=\varepsilon \frac{\zeta -\zeta^{-1}}{\zeta-\zeta^{-1}}=\varepsilon,$$
		$$S_{n-1}(x)=S_{n-1}(\zeta+\zeta^{-1})=\frac{\zeta^{n-1} -\zeta^{-n+1}}{\zeta-\zeta^{-1}}=\varepsilon \frac{\zeta^{-1} -\zeta}{\zeta-\zeta^{-1}}=-\varepsilon,$$
		and 
		$$S_n(2\varepsilon)=S_n(\varepsilon+\varepsilon^{-1})=\varepsilon^{n-1}+ \varepsilon^{n-3}+\ldots + \varepsilon^{-n+1}=n\varepsilon^{n-1}.$$
		Moreover, by Lemma \ref{lemma:Chebyshev},
		$$S_n'(x)=S_n'(\zeta+\zeta^{-1})=\frac{n(\zeta^n+\zeta^{-n})}{(\zeta-\zeta^{-1})^2}=\frac{2n\varepsilon}{x^2-4},$$
		since $\zeta^n-\zeta^{-n}=0$ and $(\zeta-\zeta^{-1})^2=\zeta^2+\zeta^{-2}-2=x^2-4.$
		Finally, 
		\begin{multline*}S_{n+1}(x)=S_{n+1}'(\zeta+\zeta^{-1})=\frac{(n+1)(\zeta^{n+1}+\zeta^{-n-1})}{(\zeta-\zeta^{-1})^2}-\frac{(\zeta^{n+1}-\zeta^{-n-1})(\zeta+\zeta^{-1})}{(\zeta-\zeta^{-1})^3}
		\\ =\frac{(n+1)x\varepsilon}{x^2-4}-\frac{x\varepsilon}{(\zeta-\zeta^{-1})^2}=\frac{nx\varepsilon}{x^2-4}.
		\end{multline*}
		The computation for $S_{n-1}'(x)$ is similar and left to the reader.
	\end{proof}
	
	\subsection{Trace functions of powers}
	\label{sec:trace_funct}
	For $\Gamma$ a group, we recall that the coordinate ring $\C[\X(\Gamma)]$ is the ring spanned by formal variables $t_x$ for $x\in \Gamma$ and with relations, for $x,y\in \Gamma,$
	$$t_{xy}+t_{xy^{-1}}=t_xt_y, t_{xy}=t_{yx}, t_1=2.$$
	The variable $t_x$ may be interpreted as the trace function
	$$t_x : [\rho] \in \X(\Gamma) \longrightarrow  \tr \rho(x).$$
	
	In the next section, we will need to make use of the following formulas:
	
	\begin{lemma}
		\label{lemma:trace-power} Let $\Gamma$ be a group, let $a,b\in \Gamma,$ and let $n\geq 0$ be an integer. Then we have the following identities in $\C[\X(\Gamma)]:$
		\begin{equation}\label{eq:powertrace1}t_{a^n}=T_n(t_a)\end{equation}
		\begin{equation}\label{eq:powertrace2}t_{a^nb}=S_n(t_a)t_{ab}-S_{n-1}(t_a)t_b \end{equation}
	\end{lemma}
	\begin{proof}
		We remark that Equation \ref{eq:powertrace1} is obviously true for $n=0$ or $1,$ and furthermore that the relation $t_{xy}+t_{xy^{-1}}=t_xt_y$ gives for $x=a^n$ and $y=a$ that $t_{a^{n+1}}=-t_{a^{n-1}}+t_at_{a^n}.$ An easy induction shows then that $t_{a^n}=T_n(t_a)$ for all $n\geq 0.$
		
		Similarly, $t_b= 0\cdot t_{ab}-(-t_b),$ and $t_{ab}=1\cdot t_{ab}-0\cdot t_b,$ so Equation \ref{eq:powertrace2} is true for $n=0$ or $1.$ If we assume by induction that $t_{a^kb}=S_k(t_a)t_{ab}-S_{k-1}(t_a)t_b$ for all $k\leq n,$ then
		
		\begin{multline*}t_{a^{n+1}b}=-t_{a^{n-1}b}+t_at_{a^nb}=(t_a S_n(t_a)-S_{n-1}(t_a))t_{ab}-(t_aS_{n-1}(t_a)-S_{n-2}(t_a))t_b
		\\=S_{n+1}(t_a)t_{ab}-S_{n}(t_a)t_b,
		\end{multline*}
		which proves Equation \ref{eq:powertrace2} for all $n.$  
	\end{proof}
	
	\section{Computation of $\C[\X(M)]$}
	\label{sec:presCarac}
	\subsection{Presentation of $\C[\X(\Gamma)]$ from a presentation of $\Gamma$}
	\label{sec:pres}
	In this section, we review the main theorem of \cite{HP23}, before applying it in the next subsection.
	
	Let $\Gamma=\langle x_1,\ldots,x_n \ | \ r_1=1,\ldots ,r_k=1 \rangle$ be a finitely presented group. Then we have a surjective morphism $F_n\twoheadrightarrow \Gamma,$ which induces a surjection $\C[\X(F_n)]\twoheadrightarrow \C[\X(\Gamma)]$ by sending the formal variable $t_x$ in $\C[\X(F_n)]$ to $t_x$ in $\C[\X(\Gamma)]$ for any word $x$ in the generators $x_i^{\pm 1}$ of $F_n.$
	
	We will describe $\C[\X(\Gamma)]$ as a quotient of $\C[\X(F_n)].$ A presentation of $\C[\X(F_n)]$ is given in \cite{ABL18} but we will only need the case $n=3,$ which we give below, and state in a compatible way to the notations used in the next section:
	
	\begin{proposition}
		\label{prop:F3-charac} \cite[Example 15]{HP23} Let $F_3=\langle h,c_1,c_2 \rangle$ be a free group with $3$ generators, and set $c_3=c_1c_2.$ Define the following elements of $\C[\X(F_3)]:$
		$$u=t_h, x_i=t_{c_i}, y_i=t_{hc_i}, \ \textrm{for} \ i=1,2,3.$$
		Then
		$$\C[\X(F_3)]=\C[u,x_1,x_2,x_3,y_1,y_2,y_3]/(F)$$
		where 
		\begin{multline*}F=y_3^2-(ux_3+x_1y_2+x_2y_1-ux_1x_2)y_3
		\\+(u^2+x_1^2+x_2^2+y_1^2+y_2^2+x_3^2+y_1y_2x_3-ux_1y_1-ux_2y_2-x_1x_2x_3-4).
		\end{multline*}
		
	\end{proposition}
	
	When $\Gamma$ is a finitely presented group, with a presentation of the form 
	$$\Gamma=\langle x_1,\ldots,x_n \ | \ w_i=w_i', \ i=1,\ldots,k \rangle$$
	with $n$ generators and $k$ relations, where the $w_i$'s and $w_i'$'s are words in $x_1,\ldots, x_n,$ \cite[Theorem 1, Addendum]{HP23} gives a presentation of $\C[\X(\Gamma)]$ as follows:
	
	\begin{theorem}
		\label{thm:pres-characScheme}\cite[Theorem 1, Addendum]{HP23} Let $\Gamma$ be a finitely presented group
		$$\Gamma=\langle x_1,\ldots,x_n \ | \ w_i=w_i', \ i=1,\ldots,k \rangle$$ then
		$$\C[\X(\Gamma)]=\C[\X(F_n)]/I$$
		where $I$ is the ideal generated by the elements $t_{w_ix}-t_{w_i'x},$ for $i\in \lbrace 1,\ldots,k \rbrace$ and for $x\in \lbrace 1 \rbrace \cup \lbrace x_i, \ 1\leq i \leq n \rbrace \cup \lbrace x_ix_j , \ 1\leq i <j \leq n \rbrace.$
	\end{theorem}
	We illustrate this theorem by computing $\C[\X(\Z\times F_2)],$ noticing that the fundamental groups of small Seifert manifolds are quotients of $\Z\times F_2.$
	
	\begin{proposition}\label{prop:ZxF2-characScheme} Let 
		$\Gamma=\Z\times F_2=\langle h,c_1,c_2 \ | hc_1=c_1h, hc_2=c_2h \rangle,$ and set $c_3=c_1c_2,$  $u=t_h,$ and $x_i=t_{c_i}, y_i=t_{hc_i},$ for $i=1,2,3.$ Then 
		$$\C[\X(\Gamma)]=\C[u,x_1,x_2,x_3,y_1,y_2,y_3]/(J)$$
		where $J$ is the ideal generated by the elements
		\begin{equation}
		\label{eq:rel1} u^2+x_i^2+y_i^2-ux_iy_i-4, \ i=1,2,3
		\end{equation}
		\begin{equation}
		\label{eq:rel2} 2y_i-(ux_i+x_jy_3-x_3y_j), \ \lbrace i,j \rbrace=\lbrace 1,2\rbrace
		\end{equation}
		\begin{equation}
		\label{eq:rel3} 2y_3-(u(x_3-x_1x_2)+y_1x_2+y_2x_1)
		\end{equation}
		\begin{equation}
		\label{eq:rel4} 2x_3-(uy_3+x_1x_2-y_1y_2)
		\end{equation}
	\end{proposition}
	\begin{proof}
		Applying Theorem \ref{thm:pres-characScheme}, the algebra $\C[\X(\Gamma)]$ is the quotient of $\C[u,x_1,x_2,x_3,y_1,y_2,y_3]$ by the relations $t_{hc_ix}=t_{c_ihx}$ for $i=1,2$ and $x\in \lbrace 1,h,c_1,c_2,hc_1,hc_2,c_1c_2 \rbrace,$ and the relation $F=0$ from Proposition \ref{prop:F3-charac}. We have to express the relations $t_{hc_ix}=t_{c_ihx}$ as polynomials in $u,x_1,x_2,x_3,y_1,y_2,y_3.$
		
		First the relations $t_{hc_i}=t_{c_ih},$ the relations $t_{hc_ih}=t_{c_ih^2}$ and the relations $t_{hc_i^2}=t_{c_ihc_i}$ are vacuous by the properties of trace functions. Next we inspect $t_{hc_ihc_i}=t_{c_ih^2c_i}$ for $i=1,2.$ An easy computation using the rules $t_{xy}+t_{xy^{-1}}=t_xt_y$ and $t_{xy}=t_{yx}$ gives that
		$$t_{hc_ihc_i}-t_{c_ih^2c_i}=u^2+x_i^2+y_i^2-ux_iy_i-4.$$
		Note that we have in $\Gamma,$ that $hc_3=c_3h$ where $c_3=c_1c_2.$ Thus the relation $t_{hc_3hc_3}=t_{c_3h^2c_3}$ and hence $u^2+x_3^2+y_3^2-ux_3y_3-4=0$ is also true in $\C[\X(\Gamma)].$ Although this equation is not part of the presentation given by Theorem \ref{thm:pres-characScheme}, we keep it for symmetry and to simplify the other relations.
		
		We get Equation \ref{eq:rel2} from the equation $t_{hc_1(c_1c_2)}=t_{c_1h(c_1c_2)}.$ We have
		\begin{multline*}
		t_{hc_1c_1c_2}=t_{c_1hc_1c_2}\Leftrightarrow t_{c_1}t_{hc_1c_2}-t_{hc_2}=t_{hc_1}t_{c_1c_2}-t_{hc_2^{-1}}
		\\ \Leftrightarrow x_1y_3-y_2
		=y_1x_3-(t_ht_{c_2}-t_{hc_2})
		\Leftrightarrow 2y_2-ux_2=x_1y_3-y_1x_3
		\end{multline*}
		which gives Equation \ref{eq:rel2} for $i=2.$ Similarly, the equation $t_{hc_2c_1c_2}=t_{c_2hc_1c_2}$ gives Equation \ref{eq:rel2} for $i=1.$
		Next we have the equations $t_{hc_1c_2}=t_{c_1hc_2}$ and $t_{hc_2c_1}=t_{c_2hc_1}.$ These two equations are actually equivalent, because, by the properties of trace functions, $t_{hc_1c_2}=t_{c_2hc_1}$ and $t_{c_1hc_2}=t_{hc_2c_1}.$ We compute
		\begin{multline*}t_{hc_2c_1}=t_ht_{c_1c_2}-t_{hc_1^{-1}c_2^{-1}}=ux_3-(x_2t_{hc_1^{-1}}-t_{hc_1^{-1}c_2})
		\\=ux_3-(x_2(ux_1-y_1))+t_{hc_1c_2}-t_{c_1}t_{hc_2}=u(x_3-x_1x_2)+y_1x_2+x_1y_2-y_3.\end{multline*}
		Thus the equation $t_{hc_1c_2}=t_{c_1hc_2}=t_{hc_1c_2}$ gives Equation \ref{eq:rel3}.
		
		Finally, we have the equations $t_{hc_ihc_j}=t_{c_ihhc_j}$ for $\lbrace i,j \rbrace=\lbrace 1,2 \rbrace.$ We claim that the two equations are equivalent. Indeed, $t_{hc_ihc_j}=t_{hc_jhc_i}.$ Moreover, $t_{c_ih^2c_j}=t_ht_{c_ihc_j}-t_{c_ic_j}.$ However, $t_{c_ic_j}=t_{c_jc_i}$ and we already assumed $t_{c_ihc_j}=t_{hc_ic_j}=t_{c_jhc_i}.$ Therefore, we need only to look at this equation for $i,j=2,1.$ We have
		$$t_{hc_2hc_1}=t_{hc_2}t_{hc_1}-t_{hc_2c_1^{-1}h^{-1}}=y_1y_2+t_{c_2c_1}-t_{c_1}t_{c_2}=y_1y_2+x_3-x_1x_2.$$
		On the other hand,
		$$t_{c_2h^2c_1}=t_ht_{c_2hc_1}-t_{c_1c_2}=ut_{hc_1c_2}-x_3=uy_3-x_3.$$
		This gives Equation \ref{eq:rel4}.
		Finally, we need to consider the element $F$ from Proposition \ref{prop:F3-charac}. Let $\delta_i$ the element corresponding to Equation \ref{eq:rel1}, and $\mu,\nu$ the elements corresponding to Equations \ref{eq:rel3} and \ref{eq:rel4}. Then a direct computation shows that
		
		$$F=\delta_1+\delta_2-\delta_3-x_3\mu-y_3\nu.$$
		Therefore, the equation $F=0$ is redundant.
	\end{proof}
	
	\subsection{A presentation of $\C[\X (M)]$ for $M$ a small Seifert manifold}
	In this section we will give a presentation of $\C[\X (M)]$ when $M$ is a Seifert manifold fibering over $S^2$ with $3$ exceptional fibers. Let us write $M=S^2(-\frac{q_1}{p_1},-\frac{q_2}{p_2},\frac{q_3}{p_3})$ where each $p_i$ is coprime with $q_i.$ By \cite{JN83}, we have the following presentation:
	\begin{equation}
	\label{eq:pres2}\pi_1(M)=\langle h,c_1,c_2 | hc_1=c_1h, hc_2=c_2h,  c_i^{p_i}=h^{q_i}, i=1,2,3 \rangle
	\end{equation}
	where in the above we set $c_3:=c_1c_2.$ Compared with the presentation given in Equation \ref{eq:pres}, we have gotten rid of the generator $c_3$ and re-arranged the relations to make it easier to apply Theorem \ref{thm:pres-characScheme}; moreover our choice of signs in the Seifert parameters is so that no sign appears in the presentation.
	
	We will get the following:
	\begin{theorem}
		\label{thm:presCharacSchemeSeifert} For $M=S^2(-\frac{q_1}{p_1},-\frac{q_2}{p_2},\frac{q_3}{p_3})$ a small Seifert manifold, the ring $\C[\X(M)]$ is isomorphic to the quotient of $\C[u,x_1,x_2,x_3,y_1,y_2,y_3]$ by the ideal generated by the elements \ref{eq:rel1}, \ref{eq:rel2}, \ref{eq:rel3} and \ref{eq:rel4} of Proposition \ref{prop:ZxF2-characScheme}, and the elements:
		\begin{equation}
		\label{eq:rel5} T_{q_i}(u)-T_{p_i}(x_i), \ i=1,2,3
		\end{equation}
		\begin{equation}
		\label{eq:rel6} T_{q_i+1}(u)-S_{p_i}(x_i)y_i+S_{p_i-1}(x_i)u, \ i=1,2,3
		\end{equation}
		\begin{equation}
		\label{eq:rel9}S_{q_i}(u)y_i-S_{q_i-1}(u)x_i-T_{p_i+1}(x_i), \ i=1,2,3
		\end{equation}
		\begin{equation}
		\label{eq:rel7} S_{q_i}(u)y_j-S_{q_i-1}(u)x_j-S_{p_i}(x_i)x_3+S_{p_i-1}(x_i)x_j, \ \lbrace i , j \rbrace=\lbrace 1,2\rbrace
		\end{equation}
		\begin{equation}
		\label{eq:rel8}
		S_{q_i+1}(u)y_i-S_{q_i}(u)x_i-S_{p_i+1}(x_i)y_i+S_{p_i}(x_i)u, \ i=1,2,3
		\end{equation}
		\begin{equation}
		\label{eq:rel10}S_{q_i}(u)y_k-S_{q_i-1}(u)x_k-S_{p_i+1}(x_i)x_k+S_{p_i}(x_i)x_j, \ i\neq j\neq k, \ \textrm{and} \ j\neq 3
		\end{equation}
		\begin{equation}
		\label{eq:rel11}S_{q_i+1}(u)y_j-S_{q_i}(u)x_j-S_{p_i}(x_i)y_3+S_{p_i-1}(x_i)y_j, \ \lbrace i,j \rbrace=\lbrace 1,2\rbrace
		\end{equation}
		
		\begin{equation}
		\label{eq:rel12} S_{q_3+1}(u)y_i-S_{q_3}(u)x_i-S_{p_3}(x_3)(x_iy_3-y_j)+S_{p_3-1}(x_3)y_i, \ \lbrace i,j \rbrace=\lbrace 1,2  \rbrace
		\end{equation}
	\end{theorem}
	
	\begin{proof}
		We apply Theorem \ref{thm:pres-characScheme} to Presentation \ref{eq:pres2}. We adopt the same notations as in the proof of Proposition \ref{prop:ZxF2-characScheme}, thus $\C[\X(\pi_1(M))]$ will be a quotient of $\C[u,x_1,x_2,x_3,y_1,y_2,y_3]$. Notice that the first two relations, $c_1h=hc_1$ and $c_2h=hc_2$ are the relations defining $F_2\times\Z.$ By Proposition \ref{prop:ZxF2-characScheme}, they will lead to relations \ref{eq:rel1}, \ref{eq:rel2}, \ref{eq:rel3} and \ref{eq:rel4}. We compute what other relations are implied by the relations $c_i^{p_i}=h^{q_i},$ which we will denote by $(r_i).$ They will correspond to the trace of $(r_i)x,$ for $x\in \lbrace 1,h,c_1,c_2,hc_1,hc_2,c_1c_2 \rbrace.$ We will repeatedly use Lemma \ref{lemma:trace-power} to simplify those traces.
		
		Taking the trace of $(r_i)$ for $i=1,2,3$ we get, by the first item of Lemma \ref{lemma:trace-power}:
		$$t_{h^{q_i}}=T_{q_i}(u)=T_{p_i}(x_i)=t_{c_i^{p_i}},$$
		that is, we get Relation \ref{eq:rel5}.
		
		Taking the trace of $(r_i)h$ for $i=1,2,3,$ we get
		$$T_{q_i+1}(u)=t_{h^{q_i+1}}=t_{hc_i^{p_i}}=S_{p_i}(t_{c_i})t_{hc_i}-S_{p_i-1}(t_{c_i})t_h=S_{p_i}(x_i)y_i-S_{p_i-1}(x_i)u,$$
		that is, we get Relation \ref{eq:rel6}. Here we used the second item of Lemma \ref{lemma:trace-power} to get the third equality.
		
		Similarly, the trace of $(r_i)c_i$ for $i=1,2,3$ gives Relation \ref{eq:rel9}. The trace of $(r_i)c_j$ for $\lbrace i,j \rbrace=\lbrace 1,2 \rbrace$ gives Relation \ref{eq:rel7}. The trace of $(r_i)hc_i$ for $i=1,2,3$ gives Relation \ref{eq:rel8}. (Note that the trace of $(r_3)hc_3$ is normally not among the relations given in Theorem \ref{thm:pres-characScheme}, but we can always include it for symmetry).
		
		The trace of $(r_i)c_3$ for $i=1$ or $2$ gives Relation \ref{eq:rel10}, in the case where $k=3$ and $\lbrace i,j \rbrace=\lbrace 1,2 \rbrace.$  On the other hand, the trace of $(r_3)c_i$ where $i\neq 3$ gives
		$$S_{q_3}(u)y_i-S_{q_3-1}(u)x_i= t_{h^{q_3}c_i}=t_{c_3^{p_3}c_i}=S_{p_3}(x_3)t_{c_3c_i}-S_{p_3-1}(x_3)x_i $$
		However, one computes from properties of trace functions that $t_{c_3c_i}=t_{c_i^2c_j}=x_ix_3-x_j,$ where $\lbrace 1,2\rbrace=\lbrace i,j \rbrace.$ Hence, we have
		$$S_{q_3}(u)y_i-S_{q_3-1}(u)x_i=S_{p_3}(x_3)(x_ix_3-x_j)-S_{p_3-1}x_i=S_{p_3+1}(x_3)x_i-S_{p_3}(x_3)x_j,$$
		that is, we get the remaining cases of Relation \ref{eq:rel10}.
		
		Next, the trace of $(r_i)hc_j$ for $\lbrace i,j \rbrace=\lbrace 1,2 \rbrace$ gives $t_{h^{q_i+1}c_j}=t_{c_i^{p_i}hc_j},$ thus using Lemma \ref{lemma:trace-power}:
		$$S_{q_i+1}(u)y_j-S_{q_i}(u)x_j=S_{p_i}(x_i)t_{hc_jc_i}-S_{p_i-1}(x_i)y_j.$$
		If $(i,j)=(2,1),$ then $t_{hc_jc_i}$ is simply $y_3.$ However, if $(i,j)=(1,2),$ then $t_{hc_2c_1}=t_{hc_1c_2}$ thanks to Equation \ref{eq:rel3}. Thus we get Equation \ref{eq:rel11}.
		
		Finally, we look at the trace of $(r_3)hc_i$ for $i=1$ or $2,$ and let $j\in \lbrace 1,2,3 \rbrace \setminus \lbrace i,3\rbrace.$ Thanks to Lemma \ref{lemma:trace-power}, we get
		$$S_{q_3+1}(u)y_i-S_{q_3}(u)x_i=S_{p_3}(x_3)t_{c_1c_2hc_i}-S_{p_3-1}(x_3)y_i.$$
		To conclude, we compute using trace functions properties that
		$$t_{c_1c_2hc_1}=t_{hc_1^2c_2}=t_{c_1}t_{hc_1c_2}-t_{hc_2}=x_1y_3-y_2$$
		and
		$$t_{c_1c_2hc_2}=t_{hc_1c_2^2}=t_{c_2}t_{hc_1c_2}-t_{hc_1}=x_2y_3-y_1.$$
		In the above, the first equality comes from Relation \ref{eq:rel2}, since it is equivalent to $t_{hc_2c_1c_2}=t_{c_2hc_1c_2}.$
		Therefore, we get
		$$S_{p_3}(x_3)t_{c_1c_2hc_i}-S_{p_3-1}(x_3)y_i=S_{p_3}(x_3)(x_iy_3-y_j)-S_{p_3-1}(x_3)y_i,$$
		for $\lbrace i,j \rbrace=\lbrace 1,2 \rbrace,$ which gives Relation \ref{eq:rel12}.
		
		Since we have now taken into account all relations in Theorem \ref{thm:pres-characScheme}, we have a presentation of $\C[\X(\pi_1(M))].$

	\end{proof}
	
	\section{The tangent space at exceptional abelian characters}
	\label{sec:tangentSpace}
	
	We will start by recalling from \cite{DKS2} the definition of exceptional abelian characters:
	
	\begin{definition}
		\label{def:exceptional} Let $M=S^2(-\frac{q_1}{p_1},-\frac{q_2}{p_2},\frac{q_3}{p_3})$ be a small Seifert manifold, with $\pi_1(M)$ given by Presentation \ref{eq:pres}. Let $\chi\in X(\pi_1(M)).$ We say that $\chi$ is an exceptional abelian character if $\chi$ is the character of an abelian representation of $\pi_1(M),$ and moreover, $\chi(u)=\pm 2,$ and $\chi(c_i)\neq \pm 2$ for $i=1,2,3.$
	\end{definition}
	In \cite[Section 5.3]{DKS2}, the following was shown:
	\begin{proposition}
		\label{prop:red-nonexcep}
		Let $M$ be a small Seifert manifold and let $\chi \in X(M).$ Then $\chi$ is isolated in $X(M),$ and moreover if $\chi$ is not an exceptional abelian character, then $\X(M)$ is reduced at $\chi.$
	\end{proposition} 
	In light of Proposition \ref{prop:red-nonexcep}, what remains to understand is whether $X(M)$ is reduced at exceptional abelian characters. Since all points in $X(M)$ are isolated, this is equivalent to checking whether $T_{\chi}\X(M)$ is non-zero for $\chi$ an exceptional abelian character.
	
	The main result of this section is the following:
	
	\begin{proposition}
		\label{prop:nred-excep} Let $M$ be a small Seifert manifold and let $\chi \in X(M)$ be an exceptional abelian character. Then $T_{\chi}\X(M)$ has dimension $1.$
	\end{proposition}
	
	In particular, we get that $\X(M)$ is reduced if and only if $M$ does not have any exceptional abelian character.

	Since we have a presentation of $\C[\X(M)],$ the tangent space at $\chi$ may be computed thanks to the following standard lemma:
	
	\begin{lemma}\label{lemma:tangentSpacePres} Let $\X$ be an affine scheme, with $\C[\X]\simeq \C[x_1,\ldots,x_k]/(f_1,\ldots ,f_r),$ and let $z$ be a closed point in $\X.$ Then $T_z\X$ admits the presentation:
		$$T_z\X\simeq \mathrm{Span}_{\C}(dx_1,\ldots,dx_k)/\mathrm{Span}_{\C}(df_1,\ldots,df_r),$$
		where the relation $df_i$ is
		$$df_i= \underset{j=1}{\overset{k}{\sum}} \frac{\partial f_i}{\partial x_j}(z) dx_j.$$ 
	\end{lemma}
	
	\begin{proof}[Proof of Proposition \ref{prop:nred-excep}]
		Let $\chi\in X(M)$ be an exceptional abelian character, and write 
		$$\chi(h)=2\varepsilon, \ \chi(c_1)=\zeta_1+\zeta_1^{-1}, \ \chi(c_2)=\zeta_2+\zeta_2^{-1}, \ \chi(c_3)=\zeta_1\zeta_2+(\zeta_1\zeta_2)^{-1},$$
		where $\varepsilon=\pm 1,$ and $\zeta_1,\zeta_2$ and $\zeta_3=\zeta_1\zeta_2$ are not $\pm 1$ and satisfy $\zeta_i^{p_i}=\varepsilon^{q_i}.$
		Note that one also has for $i=1,2,3$ that
		$$\chi(hc_i)=\varepsilon \chi(c_i).$$
		We use Theorem \ref{thm:presCharacSchemeSeifert} and Lemma \ref{lemma:tangentSpacePres} to compute $T_{\chi}\X(M).$
		The tangent space $T_{\chi}\X(M)$ will be spanned by variables $du,dx_1,dx_2,dx_3,dy_1,dy_2,dy_3,$ and subjects to relations that are obtained by differentiating relations \ref{eq:rel1} to \ref{eq:rel4}, and \ref{eq:rel5} to \ref{eq:rel12}. We will denote $\chi(h),\chi(c_i)$ and $\chi(hc_i)$ by $u, x_i,y_i$ respectively, for $i=1,2,3.$
		
		We start by differentiating Relation \ref{eq:rel5}. We get, for $i=1,2$ or $3:$
		$$T_{q_i}'(u)du=T_{p_i}'(x_i)dx_i\Leftrightarrow q_iS_{q_i}(u)du=p_iS_{p_i}(x_i)dx_i.$$
		However, by Lemma \ref{lemma:Chebyshev2}, $S_{p_i}(x_i)=0,$ while $S_{q_i}(u)\neq 0,$ so these relations are equivalent to:
		\begin{equation}
		\label{eq:diff1} du=0
		\end{equation}
		Given Equation \ref{eq:diff1}, derivating Equation \ref{eq:rel1}, and taking into account Equation \ref{eq:diff1}, we get, for $i=1,2$ or $3,$ that:
		$$2x_idx_i+2y_idy_i -ux_idy_i-uy_idx_i=0.$$
		However, $y_i=\varepsilon x_i,$ $u=2\varepsilon,$ so the above equation is equivalent to
		$$2x_i dx_i +2\varepsilon x_idy_i-2\varepsilon x_i dy_i-2\varepsilon^2 x_i dx_i=0,$$
		which is trivally satisfied since $\varepsilon^2=1.$
		
		Next we consider the derivative of Relation \ref{eq:rel9}. Taking into account that $du=0,$ we get, for $i=1,2$ or $3:$
		\begin{multline*}S_{q_i}(2\varepsilon)dy_i-S_{q_i-1}(2\varepsilon)dx_i=(p_i+1)S_{p_i+1}(x_i)dx_i 
		\\ \Leftrightarrow q_i\varepsilon^{q_i-1}dy_i-(q_i-1)\varepsilon^{q_i-2}dx_i=\varepsilon^{q_i}(p_i+1)dx_i
		\end{multline*}
		using Lemma \ref{lemma:Chebyshev2}. Multiplying by $\varepsilon^{q_i},$ we get:
		\begin{equation}
		\label{eq:diff2} \varepsilon q_idy_i=\left(p_i+q_i\right)dx_i
		\end{equation}
		Now, we check that the derivatives of Equations \ref{eq:rel6} and \ref{eq:rel8} are redundant with Equations \ref{eq:diff1} and \ref{eq:diff2}.  
		The derivative of Equation \ref{eq:rel6} is (taking into account $du=0$):
		
		$$-S_{p_i}(x_i)dy_i-S_{p_i}'(x_i)y_idx_i+S_{p_i-1}'(x_i)udx_i=0.$$
		Since $S_{p_i}(x_i)=0, S_{p_i}'(x_i)=\frac{2p_i \varepsilon^{q_i}}{x_i^2-4}, S_{p_i-1}'(x_i)=\frac{p_i x_i \varepsilon^{q_i}}{x_i^2-4}$ by Lemma \ref{lemma:Chebyshev2}, and recalling $u=2\varepsilon,$ $y_i=\varepsilon x_i,$ we get:
		$$-\frac{2 p_i \varepsilon^{q_i}}{x_i^2-4}\varepsilon x_i dx_i + \frac{p_i x_i \varepsilon^{q_i}}{x_i^2-4}(2\varepsilon)dx_i=0,$$
		which is trivially satisfied.
		
		Similarly, the derivative of Equation \ref{eq:rel8} is (given that $du=0$)
		$$S_{q_i+1}(u)dy_i-S_{q_i}(u)dx_i-S_{p_i+1}'(x_i)y_idx_i-S_{p_i+1}(x_i)dy_i+S_{p_i}'(x_i)udx_i=0$$
		Lemma \ref{lemma:Chebyshev2} gives $S_{q_i+1}(u)=(q_i+1)\varepsilon^{q_i}, S_{q_i}(u)=q_i\varepsilon^{q_i-1}, S_{p_i+1}(x_i)=\varepsilon^{q_i}, S_{p_i+1}'(x_i)=\frac{p_ix_i\varepsilon^{q_i}}{x_i^2-4},$ and $S_{p_i}'(x_i)=\frac{2p_i\varepsilon^{q_i}}{x_i^2-4}.$ Injecting in the previous equation, we get:
		\begin{multline*}
		(q_i+1)\varepsilon^{q_i}dy_i-q_i\varepsilon^{q_i-1}dx_i-\frac{p_ix_iy_i\varepsilon^{q_i}}{x_i^2-4}dx_i-\varepsilon^{q_i}dy_i+\frac{4p_i\varepsilon^{q_i-1}}{x_i^2-4}dx_i=0
		\\ \Leftrightarrow p_idx_i +\frac{p_i(4-\varepsilon x_i y_i)}{x_i^2-4}dx_i=0
		\end{multline*}
		where we used Equation \ref{eq:diff2} and simplified by $\varepsilon^{q_i-1}.$ Since $y_i=\varepsilon x_i,$ the last equation is trivially satisfied.
		
		We turn to the derivative of Equation \ref{eq:rel7}. Keeping in mind $du=0$ and $S_{p_i}(x_i)=0$ we get:
		$$S_{q_i}(u)dy_j-S_{q_i-1}(u)dx_j-S_{p_i}'(x_i)x_3 dx_i +S_{p_i-1}(x_i)dx_j+S_{p_i-1}'(x_i)x_jdx_i=0,$$	
		for $\lbrace i,j \rbrace=\lbrace 1,2 \rbrace.$ However, by Lemma \ref{lemma:Chebyshev2}, this simplifies to
		$$q_i\varepsilon^{q_i-1}dy_j-(q_i-1)\varepsilon^{q_i-2}dx_j-\frac{2p_i\varepsilon^{q_i}x_3}{x_i^2-4}dx_i-\varepsilon^{q_i}dx_j+\frac{p_ix_ix_j\varepsilon^{q_i}}{x_i^2-4}dx_i=0$$
		Now, we use $\varepsilon q_jdy_j=(p_j+q_j)dx_j$ and simplify by $\varepsilon^{q_i}:$
		$$\frac{q_ip_j}{q_j}dx_j-\frac{2p_ix_3}{x_i^2-4}dx_i+\frac{p_ix_ix_j}{x_i^2-4}dx_i=0\Leftrightarrow \frac{p_j}{q_j}dx_j=\frac{(2x_3-x_ix_j)p_i}{(x_i^2-4)q_i}dx_i$$
		Now, recall that $x_i=\zeta_i+\zeta_i^{-1}$ for $i=1,2$ and $x_3=\zeta_1\zeta_2+(\zeta_1\zeta_2)^{-1}.$ Substituting we find that 
		\begin{equation}
		\label{eq:diff3}\frac{p_jdx_j}{q_j(\zeta_j-\zeta_j^{-1})}=\frac{p_idx_i}{q_i(\zeta_i-\zeta_i^{-1})},
		\end{equation}
		where $\lbrace i,j\rbrace=\lbrace 1,2\rbrace.$
		Because of Equation \ref{eq:diff3}, let us introduce the notation
		$$dX=\frac{p_1dx_1}{q_1(\zeta_1-\zeta_1^{-1})}=\frac{p_2dx_2}{q_2(\zeta_2-\zeta_2^{-1})}.$$
		We turn to the derivative of Equation \ref{eq:rel4}. We have, using $du=0:$
		$$2dx_3-udy_3-x_1dx_2-x_2dx_1+y_1dy_2+y_2dy_1=0$$
		Using Equation \ref{eq:diff2} and $u=2\varepsilon, y_i=\varepsilon x_i,$ this simplifies to:
		$$2\frac{p_3}{q_3}dx_3=\frac{p_1}{q_1}x_2dx_1+\frac{p_2}{q_2}x_1dx_2.$$
		Finally, using Equation \ref{eq:diff3}, we get:
		$$2\frac{p_3}{q_3}dx_3=(\zeta_2+\zeta_2^{-1})(\zeta_1-\zeta_1^{-1})dX +(\zeta_1+\zeta_1^{-1})(\zeta_2-\zeta_2^{-1})dX,$$
		from which we deduce the equation:
		\begin{equation}
		\label{eq:diff4}
		dX=\frac{p_1dx_1}{q_1(\zeta_1-\zeta_1^{-1})}=\frac{p_2dx_2}{q_2(\zeta_2-\zeta_2^{-1})}=\frac{p_3dx_3}{q_3(\zeta_3-\zeta_3^{-1})}
		\end{equation}
		We claim that the derivatives of all remaining equations, that is, Equations \ref{eq:rel2}, \ref{eq:rel3}, \ref{eq:rel10}, \ref{eq:rel11} and \ref{eq:rel12}, are redundant with Equations \ref{eq:diff1}, \ref{eq:diff2} and \ref{eq:diff4}. Let us start with the derivative of Equation \ref{eq:rel2}. Bearing in mind $du=0$ and $y_i=\varepsilon x_i,$ we get, for $\lbrace i,j \rbrace=\lbrace 1,2\rbrace:$
		$$2dy_i-2\varepsilon dx_i-x_jdy_3-\varepsilon x_3 dx_j+x_3dy_j+\varepsilon x_j dx_3=0.$$
		Using Equation \ref{eq:diff2} and simplifying by $\varepsilon$ this translates to:
		$$\frac{2p_idx_i}{q_i}-\frac{p_3x_jdx_3}{q_3}+\frac{p_jx_3dx_j}{q_j}=0.$$
		Now, using Equation \ref{eq:diff4} and $x_i=\zeta_i+\zeta_i^{-1}, x_j=\zeta_j+\zeta_j^{-1},x_3=\zeta_i\zeta_j+(\zeta_i\zeta_j)^{-1}:$
		$$2(\zeta_i-\zeta_i^{-1})dX-(\zeta_j+\zeta_j^{-1})(\zeta_i\zeta_j-\zeta_i^{-1}\zeta_j^{-1})dX +(\zeta_i\zeta_j+(\zeta_i\zeta_j)^{-1})(\zeta_j-\zeta_j^{-1})dX=0.$$
		Expanding the left hand side, one sees that this relation is trivially satisfied, hence the derivative of Equation \ref{eq:rel2} is redundant with Equation \ref{eq:diff1},\ref{eq:diff2} and \ref{eq:diff4}.
		
		The treatment of Equation \ref{eq:rel3} is fairly similar to the one of Equation \ref{eq:rel2}; hence we skip this calculation and leave it as exercise for the interested reader.
		
		We move to the derivative of Equation \ref{eq:rel10}. Using $du=0$ and Lemma \ref{lemma:Chebyshev2}, that equation is equivalent to:
		$$q_i\varepsilon^{q_i-1}dy_k-(q_i-1)\varepsilon^{q_i}dx_k-\varepsilon^{q_i}dx_k-\frac{p_ix_ix_k\varepsilon^{q_i}dx_i}{x_i^2-4}+\frac{2p_ix_j\varepsilon^{q_i}}{x_i^2-4}=0$$
		for $i\neq j \neq k$ and $j\neq 3.$
		Next we simplify by $\varepsilon^{q_i}$ and use Equation \ref{eq:diff2} to get:
		$$\frac{q_ip_kdx_k}{q_k}+\frac{p_i(2x_j-x_ix_k)dx_i}{x_i^2-4}=0.$$
		Plugging in Equation \ref{eq:diff3}, this is equivalent to:
		$$(\zeta_k-\zeta_k^{-1})dX+\frac{2(\zeta_j+\zeta_j^{-1})-(\zeta_i+\zeta_i^{-1})(\zeta_k+\zeta_k^{-1})}{(\zeta_i-\zeta_i^{-1})}dX=0$$
		$$\Leftrightarrow \left[-2\zeta_i\zeta_k^{-1}-2\zeta_i^{-1}\zeta_k+2\zeta_j+2\zeta_j^{-1}\right]dX=0.$$
		Since $j\neq 3$ and $i\neq j \neq k,$ it is easy to see that the left hand side vanishes, hence the derivative of Equation \ref{eq:rel10} is redundant.
		
		Finally, we treat the derivative of Equation \ref{eq:rel11}. Using Lemma \ref{lemma:Chebyshev2} and Equation\ref{eq:diff1} we get:
		$$(q_i+1)\varepsilon^{q_i}dy_j-q_i\varepsilon^{q_i-1}dx_j-\frac{2p_i\varepsilon^{q_i}}{x_i^2-4}y_3dx_i+\varepsilon^{q_i-1}dy_j+\frac{p_ix_i\varepsilon^{q_i}}{x_i^2-4}y_jdx_i=0$$
		Simplifying by $\varepsilon^{q_i}$ and using Equation \ref{eq:diff2}, this is equivalent to:
		$$\frac{p_jdx_j}{q_j}+\frac{p_i(x_ix_j-2x_3)dx_i}{q_i}=0.$$
		Using Equation \ref{eq:diff4} and $x_l=\zeta_l+\zeta_l^{-1}$ for $l=i,j,3,$ with $\zeta_3=\zeta_i\zeta_j,$ this translates to:
		$$\left[ (\zeta_i-\zeta_i^{-1})(\zeta_j-\zeta_j^{-1})+(\zeta_i+\zeta_i^{-1})(\zeta_j+\zeta_j^{-1})-2(\zeta_i\zeta_j+(\zeta_i\zeta_j)^{-1})\right]dX=0.$$ 
		Expanding the left hand side reveals that this equation is trivially satisfied, that is, the derivative of Equation \ref{eq:rel11} is redundant. The case of Equation \ref{eq:rel12} is very similar; we leave it again as exercise for the interested reader.
		Hence we are done; the tangent space is the vector space spanned by the variables $du,dx_1,dx_2,dx_3,dy_1,dy_2,dy_3$ and subject to the relations \ref{eq:diff1}, \ref{eq:diff2} and \ref{eq:diff4}, it is manifest that it is one-dimensional (generated by $dx_1$ for instance).

	\end{proof}

\section{Order 2 deformations}

In this last section, we prove the following proposition, which,
together with Proposition \ref{prop:nred-excep}, concludes
the proof of Theorem \ref{thm:main}.

\begin{proposition}\label{prop:no-order-2-deformations}
Let $M$ be a small Seifert manifold and let
$\chi\in X(M)$ be an exceptional abelian character. Then $\chi$ has no
order $2$ deformations in $X(M)$.
\end{proposition}
Before we embark on the proof, let us comment how to compute order 2 deformations. For an affine subscheme $X$ of $\mathbb{A}^n_{\mathbb{C}}$ cut out by equations $(f_i(x_1,\ldots,x_n)=0)_{i=1,\ldots,m}$ a point $x\in X$ admits order $2$ deformations if and only if there are solutions to the equations
$$f_i(x+tv+t^2w)=0$$
where $v,w\in \C^n$ and $v\neq 0$ and $t$ is a formal variable with $t^3=0.$
Note that in particular, $v$ must belong to the tangent space $T_xX.$
\begin{proof}
By Proposition \ref{prop:nred-excep}, the tangent space
$T_\chi X(M)$ has dimension $1$, and furthermore, by its proof, a non-zero vector
\[
v=(du,dx_1,dx_2,dx_3,dy_1,dy_2,dy_3)\in T_\chi X(M)
\]
satisfies
$$
du=0, \ \ 
\epsilon q_idy_i=(p_i+q_i)dx_i,$$
and, writing $\zeta_3=\zeta_1\zeta_2$,
$$
dX=
\frac{p_1dx_1}{q_1(\zeta_1-\zeta_1^{-1})}
=
\frac{p_2dx_2}{q_2(\zeta_2-\zeta_2^{-1})}
=
\frac{p_3dx_3}{q_3(\zeta_3-\zeta_3^{-1})}.
$$
In particular, $dx_i\ne0$ for every $i$.

Suppose that $\chi$ admits an order $2$ deformation. Write its
coordinates, to second order, as
\[
u(t)=u+tdu+\frac{t^2}{2}u_2,\qquad
x_i(t)=x_i+tdx_i+\frac{t^2}{2}x_{i,2},
\]
and similarly for $y_i(t)$. 

Recall that by Equation~\ref{eq:rel12} we have:
\[
T_{q_i}(u)-T_{p_i}(x_i)=0.
\]
Since $du=0$ and $T'_{p_i}(x_i)=p_iS_{p_i}(x_i)=0$ at $\chi$, its
second-order term gives
\[
T'_{q_i}(u)u_2-T''_{p_i}(x_i)(dx_i)^2=0.
\]
Using $u=2\epsilon$ and
\[
T'_{q_i}(2\epsilon)=q_i^2\epsilon^{q_i-1},
\qquad
T''_{p_i}(x_i)
=p_iS'_{p_i}(x_i)
=\frac{2p_i^2\epsilon^{q_i}}{x_i^2-4},
\]
we obtain
\begin{equation}
u_2=
\frac{2p_i^2\epsilon}{q_i^2(x_i^2-4)}(dx_i)^2.
\label{eq:secondu}
\end{equation}

We next consider Equation~\ref{eq:rel7},
\[
u^2+x_i^2+y_i^2-ux_iy_i-4=0.
\]
 Since $du=0$, its second-order term is
\[
2(dx_i)^2+2(dy_i)^2
-2u\,dx_i\,dy_i-x_iy_i u_2=0.
\]
Using
\[
u=2\epsilon,\qquad y_i=\epsilon x_i,
\qquad
dy_i=\epsilon \frac{p_i+q_i}{q_i}dx_i,
\]
and Equation \ref{eq:secondu}, we find
$$\left( 2 +2\frac{(p_i+q_i)^2}{q_i^2}-2\frac{p_i+q_i}{q_i}-2\frac{x_i^2p_i^2}{q_i^2(x_i^2-4)}\right)(dx_i)^2=0.
$$
or, after simplification:
$$-\frac{8p_i^2}{q_i^2(x_i^2-4)}(dx_i)^2=0$$
contradicting the fact that $dx_i\neq 0$. Hence the exceptional abelian character $\chi$ admits no order $2$ deformation.
\end{proof}

\bibliographystyle{hamsalpha}

\bibliography{biblio}

\end{document}